\UseRawInputEncoding
\documentclass[12pt]{article}
\usepackage[centertags]{amsmath}
\usepackage{amsfonts}
\usepackage{amssymb}
\usepackage{latexsym}
\usepackage{amsthm}
\usepackage{newlfont}
\usepackage{graphicx}
\usepackage{listings}
\usepackage{booktabs}
\usepackage{abstract}
\usepackage{enumerate}
\usepackage{xcolor}
\RequirePackage{srcltx}
\date{}

\newlength{\defbaselineskip}
\newcommand{\setlinespacing}[1]%
           {\setlength{\baselineskip}{#1 \defbaselineskip}}

\newcommand{\N}{{\mathbb{N}}}

\newcommand{\actaqed}{\hfill $\actabox$}
{\medskip\noindent \textit{Proof of #1. }}%
{\actaqed \medskip}

\def\cA{{\mathcal A}}
\def\C{{\mathcal C}}
\def\cC{{\mathcal C}}
\def\cV{{\mathcal V}}
\def\cF{{\mathcal F}}

\def \Tr{\mathcal T}
\def \cK{\mathcal K}

\def \cR{\mathcal R}

\def \cX{\mathcal X}

\def\R{{\mathbb R}}
\def\Z{\mathbb Z}

\def \T{\mathbb T}

\def\bbC{\mathbb C}
\def \<{\langle}
\def\>{\rangle}

\def \Og{\Omega}

\def\bx{\mathbf x}
\def\by{\mathbf y}

\def\bk{\mathbf k}

\def\bw{\mathbf w}

\def\bs{\mathbf s}
\def\br{\mathbf r}
\def\bj{\mathbf j}

\def\ba{\mathbf a}

\def\bF{\mathbf F}

\newtheorem{Theorem}{Theorem}[section]
\newtheorem{Lemma}{Lemma}[section]
\newtheorem{Definition}{Definition}[section]
\newtheorem{Proposition}{Proposition}[section]
\newtheorem{Remark}{Remark}[section]

\numberwithin{equation}{section}

\newcommand{\be}{\begin{equation}}
\newcommand{\ee}{\end{equation}}

\def\Og{\Omega}

\def\cA{\mathcal{A}}

\def\cX{\mathcal{X}}

\def\bx{\mathbf{x}}

\def\by{\mathbf{y}}

\def\bh{\mathbf{h}}

\def\bw{\mathbf{w}}

\def\bx{\mathbf{x}}

\DeclareFontEncoding{FMS}{}{}
\DeclareFontSubstitution{FMS}{futm}{m}{n}
\DeclareFontEncoding{FMX}{}{}
\DeclareFontSubstitution{FMX}{futm}{m}{n}
\DeclareSymbolFont{fouriersymbols}{FMS}{futm}{m}{n}
\DeclareSymbolFont{fourierlargesymbols}{FMX}{futm}{m}{n}
\DeclareMathDelimiter{\VT}{\mathord}{fouriersymbols}{152}{fourierlargesymbols}{147}

\def\bh{\mathbf{h}}

\begin{document}

\title{On universal sampling recovery in the uniform norm}

\author{   V.N. Temlyakov 	\footnote{
		This research was supported by the Russian Science Foundation (project No. 23-71-30001)
at the Lomonosov Moscow State University.
  }}

\newcommand{\Addresses}{{
  \bigskip
  \footnotesize


  V.N. Temlyakov, \textsc{Steklov Mathematical Institute of Russian Academy of Sciences, Moscow, Russia; \\Lomonosov Moscow State University; \\ Moscow Center of Fundamental and Applied Mathematics; \\ University of South Carolina.
  \\
E-mail:} \texttt{temlyakovv@gmail.com}

}}
\maketitle

\begin{abstract}{It is known that results on universal sampling discretization of the square norm are useful in sparse sampling recovery with error measured in the square norm. In this paper we demonstrate how known results on universal sampling discretization of the uniform norm and recent results on universal sampling representation allow us to provide good universal methods of sampling recovery for anisotropic Sobolev and Nikol'skii classes of periodic functions of several variables. 
The sharpest results are obtained in the case of functions on two variables, where the Fibonacci point sets are used for recovery. 

	 }\end{abstract}


{\it MSC classification 2000:} Primary 65J05; Secondary 42A05, 65D30, 41A63.

\section{Introduction}
\label{I}

The idea of universal approximation and universal cubature formulas is well known in approximation 
theory. This idea was explicitly formulated and developed in \cite{VT36} and \cite{VT42}. This concerns 
approximation of smooth multivariate functions. The
concept of smoothness becomes more complicated in the multivariate case than it
is in the univariate case. In the multivariate case a function may have
different smoothness properties in different coordinate directions. In other
words, functions may belong to different anisotropic smoothness classes (see anisotropic Sobolev and 
 Nikol'skii 
classes  $W^\br_{q,\alpha}$  and $H^\br_q$ in Section \ref{ac}). It is known (see Chapter 3 of \cite{VTbookMA}) that approximation characteristics of
anisotropic smoothness classes depend on the average smoothness $g(\br)$ and optimal
approximation methods depend on anisotropy of classes, on the vector $\br$. This motivated a study in
\cite{VT36} of existence of an approximation method that is good for all
anisotropic smoothness classes. This is a problem of existence of a universal
method of approximation.  We note that the universality concept in learning
theory is very important and it is close to the concepts of adaptation and
distribution-free estimation in non-parametric statistics (\cite{GKKW},
\cite{BCDDT}, \cite{VT113}). 

The problem of finding universal methods of approximation can be
raised in the following way. Assume that we know that the function $f$
belongs, for example, to the Nikol'skii class $H^\br_q$ of periodic functions
but the vector $\br$
is not known exactly and we only know  that $\br\in P:=\prod_{j=1}^d
[A_j,B_j]$. Which
is the most natural form of the partial sums of the Fourier series
for approximation of the function $f(\mathbf x)$? 
It is proved in \cite{VT36} (see also \cite{VTbookMA}, Section 5.4.1) 
that the answer to the above question gives the hyperbolic cross
polynomials. It is proved there that in the sense of widths (orthowidth  and Kolmogorov width)
in order to universally achieve optimal errors, say, in terms of the Kolmogorov width $d_m$ we 
need to use subspaces of dimension $m(\log m)^{d-1}$ ($d$ is the number of variables). 

In this paper we consider the problem of universal sampling recovery in the uniform norm of 
periodic functions from anisotropic Sobolev and  Nikol'skii classes. It turns out 
that there exists a universal sampling recovery algorithm (nonlinear), which uses the number of points of order $m$ and provides optimal rate of sampling recovery with $m$ points for each anisotropic class. This means that in this case the use of nonlinear method allows us to build a universal 
method without loosing an extra factor $(\log m)^{d-1}$ in the number of parameters.

\section{Universal discretization and sampling recovery}
\label{ud}

 We now give explicit formulations of the
sampling discretization problem (also known as the Marcinkiewicz discretization problem) and of the problem of universal discretization. Let $\Omega$ be a compact subset of $\R^d$ with the probability measure $\mu$. By the $L_q$ norm, $1\le q< \infty$, of a function defined on $\Omega$,  we understand
$$
\|f\|_q:=\|f\|_{L_q(\Omega,\mu)} := \left(\int_\Omega |f|^qd\mu\right)^{1/q}.
$$
By the $L_\infty$ norm we understand the uniform norm of continuous functions
$$
\|f\|_\infty := \max_{\bx\in\Omega} |f(\bx)|
$$
and with a little abuse of notations we sometimes write $L_\infty(\Omega)$ for the space $\C(\Omega)$ of continuous functions on $\Omega$. In this paper we focus on the case 
$\Omega = \T^d := [0,2\pi]^d$ and $\mu$ is the normalised Lebesgue measure on $\T^d$.

{\bf The sampling discretization problem.} Let $(\Omega,\mu)$ be a probability space and 
  $X_N\subset L_q$ be an $N$-dimensional subspace of $L_q(\Omega,\mu)$ with  $1\le q \le \infty$  (the index $N$ here, usually, stands for the dimension of $X_N$).  We shall always assume that every function in $X_N$ is defined everywhere on $\Og$, and 
  \[ f\in X_N, \  \|f\|_q =0\implies f=0\in X_N.\]  We say that $X_N$  admits the Marcinkiewicz-type discretization theorem with parameters $m\in \N$ and $q$ and positive constants $C_1\le C_2$ if there exists a set $\xi := \{\xi^j\}_{j=1}^m \subset \Omega$
 such that for any $f\in X_N$ we have in the case $1\le q <\infty$
\be\label{A1}
C_1\|f\|_q^q \le \frac{1}{m} \sum_{j=1}^m |f(\xi^j)|^q \le C_2\|f\|_q^q
\ee
and in the case $q=\infty$ 
$$
C_1\|f\|_\infty \le \max_{1\le j\le m} |f(\xi^j)| \le  \|f\|_\infty.
$$

{\bf The  problem of universal discretization.} Let $\cX:= \{X(n)\}_{n=1}^k$ be a collection of finite-dimensional  linear subspaces $X(n)$ of the $L_q(\Omega)$, $1\le q \le \infty$. We say that a set $\xi:= \{\xi^j\}_{j=1}^m \subset \Omega $ provides {\it universal discretization} for the collection $\cX$ if, in the case $1\le q<\infty$, there are two positive constants $C_i$, $i=1,2$, such that for each $n\in\{1,\dots,k\}$ and any $f\in X(n)$ we have
\[
C_1\|f\|_q^q \le \frac{1}{m} \sum_{j=1}^m |f(\xi^j)|^q \le C_2\|f\|_q^q.
\]
In the case $q=\infty$  for each $n\in\{1,\dots,k\}$ and any $f\in X(n)$ we have
\be\label{1.2u}
C_1\|f\|_\infty \le \max_{1\le j\le m} |f(\xi^j)| \le  \|f\|_\infty.
\ee 
Note that the  problem of universal discretization for the collection $\cX:= \{X(n)\}_{n=1}^k$ is 
the sampling discretization problem for the set $\cup_{n=1}^k X(n)$.

We refer the reader to the survey papers \cite{DPTT} and \cite{KKLT} for results on sampling 
discretization, to the paper \cite{KKT} for recent results on sampling discretization of the uniform norm, and to \cite{VT160}, \cite{DT}, \cite{DTM1}, \cite{DTM2} for results on universal sampling discretization. 

In this paper we focus on the case $q=\infty$. 
We begin our discussion with a conditional result from \cite{VT183}. We only present the case $q=\infty$ here. 
Let $X_N$ be an $N$-dimensional subspace of the space of continuous functions $\C(\Omega)$. For a fixed $m$ and a set of points  $\xi:=\{\xi^\nu\}_{\nu=1}^m\subset \Omega$ we associate with a function $f\in \C(\Omega)$ the vector
$$
S(f,\xi) := (f(\xi^1),\dots,f(\xi^m)) \in \bbC^m.
$$
Denote
$$
\|S(f,\xi)\|_q:= \left(\frac{1}{m}\sum_{\nu=1}^m |f(\xi^\nu)|^q\right)^{1/q},\quad 1\le q<\infty,
$$
and 
$$
\|S(f,\xi)\|_\infty := \max_{\nu}|f(\xi^\nu)|.
$$
Define the best approximation of $f\in L_q(\Omega,\mu)$, $1\le q\le \infty$ by elements of $X_N$ as follows
$$
d(f,X_N)_q := \inf_{u\in X_N} \|f-u\|_q.
$$
It is well known that there exists an element, which we denote $P_{X_N,q}(f)\in X_N$, such that
$$
\|f-P_{X_N,q}(f)\|_q = d(f,X_N)_q.
$$
The operator $P_{X_N,q}: L_q(\Omega,\mu) \to X_N$ is called the Chebyshev projection. 

  Theorem \ref{AT1} below was proved in \cite{VT183} under the following assumption.

{\bf A1. Discretization.}   Suppose that $\xi:=\{\xi^j\}_{j=1}^m\subset \Omega$ is such that for any 
$u\in X_N$  we have
$$
C_1\|u\|_\infty \le \|S(u,\xi)\|_{\infty}  
$$
with a positive constant $C_1$. 

Consider the following well known recovery operator (algorithm)  
$$
\ell q(\xi)(f) := \ell q(\xi,X_N)(f):=\text{arg}\min_{u\in X_N} \|S(f-u,\xi)\|_{q}.
$$
We only consider the case $q=\infty$ here and for brevity we drop $\infty$ from the notation:
$\ell (\xi,X_N) := \ell \infty(\xi,X_N)$.

\begin{Theorem}[\cite{VT183}]\label{AT1} Under assumption {\bf A1}  for any $f\in \C(\Omega)$ we have
$$
\|f-\ell (\xi,X_N)(f)\|_\infty \le (2C_1^{-1} +1)d(f, X_N)_\infty.
$$
\end{Theorem}

We prove   the following two conditional theorems. We define a new algorithm, which is an $L_\infty$ version of the algorithm studied in \cite{DTM1} in the case of $L_2$:
$$
n(\xi,f) := \text{arg}\min_{ 1\le n\le k}\|f-\ell (\xi,X(n))(f)\|_\infty,
$$
\be\label{I4}
  \ell (\xi,\cX)(f):= \ell (\xi,X(n(\xi,f)))(f).
\ee

  \begin{Definition}\label{ID1}   We say that a set $\xi:= \{\xi^j\}_{j=1}^m \subset \Omega $ provides {\it   $L_\infty$-universal discretization}   for the collection $\cX:= \{X(n)\}_{n=1}^k$ of finite-dimensional  linear subspaces $X(n)$ if we have (for $D\ge 1$)
 \be\label{I3}
 \|f\|_\infty \le D  \max_{1\le j\le m} |f(\xi^j)|  \quad \text{for any}\quad f\in \bigcup_{n=1}^k X(n) .
\ee
We denote by $m(\cX,D)$ the minimal $m$ such that there exists a set $\xi$ of $m$ points, which
provides  $L_\infty$-universal discretization (\ref{I3}) for the collection $\cX$. 
\end{Definition}

 \begin{Theorem}\label{udT1} Let $m \in \N$ and $\cX$ be a collection of finite-dimensional subspaces.    Assume that  there exists a set $\xi:= \{\xi^j\}_{j=1}^m \subset \Omega $, which provides {\it   $L_\infty$-universal discretization}  (\ref{I3}) for the collection $\cX$. Then for   any  function $ f \in \C(\Omega)$ we have
 \be\label{I6}
  \|f-\ell (\xi,\cX)(f)\|_\infty \le  (2D +1) \min_{1\le n\le k} d(f,X(n))_\infty.
 \ee
 \end{Theorem}

\begin{proof}  
Suppose that   a set $\xi:= \{\xi^j\}_{j=1}^m \subset \Omega $ provides   $L_\infty$-universal discretization (\ref{I3}) for the collection $\cX$. Then condition {\bf A1} is satisfied for all $X(n)$ from the collection $\cX$ with $C_1=D^{-1}$.  Thus, we can apply Theorem \ref{AT1} for each 
subspace $X(n)$ with the same set of points $\xi$. It gives for all $n=1,\dots,k$ 
\be\label{A3}
\|f-\ell (\xi,X(n))(f)\|_\infty \le  (2D +1)d(f, X(n))_\infty.
\ee
 Then, inequality (\ref{A3}) and the definition (\ref{I4}) imply
\be\label{A4}
 \|f-\ell (\xi,\cX)(f)\|_\infty \le   (2D +1)\min_{1\le n\le k}d(f, X(n))_\infty.
 \ee
 This proves inequality (\ref{I6}) of Theorem \ref{udT1}.  
 \end{proof}
 
 We now formulate a direct corollary of Theorem \ref{udT1} for function classes. Denote by
 $\cA(m,k,D)$   the family of all collections $\cX:= \{X(n)\}_{n=1}^k$ of finite-dimensional  linear subspaces $X(n)$  of the $\cC(\Omega)$ such that for each $\cX$  there exists a set $\xi:= \{\xi^j\}_{j=1}^m \subset \Omega $, which provides {\it   $L_\infty$-universal discretization} (\ref{I3}) for the $\cX$. Define the following recovery characteristic (in the case $p=\infty$), which was introduced in \cite{DTM1} in the case $p=2$.
$$
\varrho^{\infty}_{m}(\bF,\cX,L_\infty) := \inf_{\{\xi^1,\dots,\xi^m\} \subset\Og} \sup_{f\in \bF} \min_{L\in\cX} \|f-\ell(\xi,L)(f)\|_\infty.
$$

For a compact subset $\bF$ and a subspace $Y$ of $\C(\Omega)$ define
$$
d(\bF,Y)_\infty := \sup_{f\in \bF} \inf_{y\in Y} \|f-y\|_\infty.
$$
 
  \begin{Theorem}\label{udT2} Let $m \in \N$ and let $\cX$ be a collection of finite-dimensional subspaces. Assume that    $\cX  \in \cA(m,k,D)$. Then for   any compact subset $\bF$ of $\C(\Omega)$,  we have  
 \be\label{I8}
 \varrho_{m}^{\infty}(\bF,\cX,L_\infty(\Omega)) \le  (2D +1)\min_{1\le n\le k}d(\bF,X(n))_\infty.
 \ee
 \end{Theorem}
 
 We refer the reader to the recent papers \cite{JUV}, \cite{DTM1}, and \cite{DTM2} for results 
in the style of Theorems \ref{udT1} and \ref{udT2}.

 \section{Universal sampling recovery for anisotropic classes}
 \label{ac} 
 
 We begin with known results on discretization from \cite{VT160}.
 We studied the  universal discretization for subspaces of the 
trigonometric polynomials in \cite{VT160}. Let $Q$ be a finite subset of $\Z^d$. We denote
$$
\Tr(Q):= \{f: \T^d\to \mathbb C: f(\bx)=\sum_{\bk\in Q}c_\bk e^{i(\bk,\bx)}\},
$$
where $\T^d := [0,2\pi]^d$. 

In \cite{VT160} we were primarily interested in the universal discretization  for the collection of subspaces of trigonometric polynomials with frequencies from parallelepipeds (rectangles). For $\bs=(s_1,\dots,s_d)\in\Z^d_+$
define
$$
R(\bs) := \{\bk =(k_1,\dots,k_d)\in \Z^d :   |k_j| < 2^{s_j}, \quad j=1,\dots,d\}.
$$
   Consider the collection $H(n,d):= \{\Tr(R(\bs)): \|\bs\|_1=n\}$. 

We proved in \cite{VT160} the following result.
\begin{Theorem}[\cite{VT160}]\label{acT1}   For every $1\le q\le\infty$ there exists a positive constant $C(d,q)$, which depends only on $d$ and $q$, such that for any $n\in \N$ there is a set $\xi(m):=\{\xi^\nu\}_{\nu=1}^m\subset \T^d$, with $m\le C(d,q)2^n$ that provides universal discretization in $L_q$   for the collection $H(n,d)$.
\end{Theorem}
 
Theorem \ref{acT1} basically solves the universal discretization problem for the collection $H(n,d)$. It provides the upper bound  $m\le C(d,q)2^n$ with $2^n$ being of the order of the dimension of each $\Tr(R(\bs))$ from the collection $H(n,d)$.
 Obviously, the lower bound for the cardinality of a set, providing the Marcinkiewicz discretization theorem for $\Tr(R(\bs))$ with $\|\bs\|_1=n$, is $\ge C(d)2^n$. 
In \cite{VT160} we treated separately the case $q=\infty$   and the case $1\le q<\infty$. Our construction of the universal set was based on deep results on 
 existence of special nets, known as $(t,r,d)$-nets. We present the definition of these important nets. 
 
 \begin{Definition}\label{acD1} A $(t,r,d)$-net (in base $2$) is a set $T$ of $2^r$ points in 
$[0,1)^d$ such that each dyadic box $[(a_1-1)2^{-s_1},a_12^{-s_1})\times\cdots\times[(a_d-1)2^{-s_d},a_d2^{-s_d})$, $1\le a_j\le 2^{s_j}$, $j=1,\dots,d$, of volume $2^{t-r}$ contains exactly $2^t$ points of $T$.
\end{Definition}
 A construction of such nets for all $d$ and $t\ge Cd$, where $C$ is a positive absolute constant, $r\ge t$  is given in \cite{NX}. 
 
 Theorems \ref{acT1} and \ref{udT1} imply the following statement for the collection $H(n,d)$.
 
  \begin{Proposition}\label{acP1} Let $n,d \in \N$.  There exists a set $\xi:= \{\xi^j\}_{j=1}^m \subset \T^d $, with $m \le C(d)2^n$, which provides    $L_\infty$-universal discretization (\ref{I3}) with $D=D(d)$ for the collection $H(n,d)$ and for   any  function $ f \in \C(\Omega)$ we have
 \be\label{ac1}
  \|f-\ell (\xi,H(n,d))(f)\|_\infty \le  (2D +1) \min_{\bs: \|\bs\|_1=n} d(f,\Tr(R(\bs)))_\infty.
 \ee
 \end{Proposition}
 
 We now apply Proposition \ref{acP1} to Sobolev and Nikol'skii anisotropic classes. We need some 
 standard definitions for that. Denote for $r>0$ and $\alpha \in \R$
 $$
 F_r(x,\alpha):= 1 + 2 \sum_{k=1}^\infty k^{-r} \cos(kx-\alpha \pi/2)
 $$
 the Bernoulli kernels.
 
 The Sobolev class $W_{q,\alpha}^{\br}    B$,
$\br = (r_1 ,\dots,r_d)$, $r_j  > 0$, $j=1,\dots,d$,
$1\le q \le \infty$ and $\alpha \in \R$
consists of functions  $f(\bx)$,
which have the following integral representation for each $1\le j\le d$
$$
f(\bx)= (2\pi)^{-1}\int_0^{2\pi}\varphi_j (x_1 ,\dots,x_{j-1},
y,x_{j+1},\dots,x_d) F_{r_j}  (x_j  - y,\alpha_j)dy ,
$$
\be\label{b3.1}
\|\varphi_j\|_{q}\le B .
\ee
 
The Nikol'skii  class  $H_{q}^{\br} B$,
$\br  =  (r_1 ,\dots,r_d)$, $r_j  > 0$, $j=1,\dots,d$, and  $1\le q  \le 
\infty$ is the set of functions $f\in L_{q}$
such that for each
$l_j  := [r_j ] + 1$, $j = 1,\dots,d$ the following relations hold
$$
\|f\|_{q}\le B ,\qquad\|\Delta_h^{l_j,j}f\|_{q}\le B|h|^{r_j},
\qquad j = 1,\dots,d ,
$$
where $\Delta_h^{l,j}$    is the $l$-th difference with step $h$ in the
variable $x_j$.  In
the case $B = 1$ we shall not  write  it  in  the  notations  of  the
Sobolev and Nikol'skii classes.
It is usual to call these classes isotropic in the case $\br =
 r \mathbf 1$, and anisotropic in the general case.
 
 It is convenient to use the following notation
 $$
 g(\br) := \left(\sum_{j=1}^d \frac{1}{r_j}\right)^{-1}.
 $$
 It is known (see, for instance \cite{VTbookMA}, p.108, Theorem 3.4.7) that for each $\br$ and $1\le q\le \infty$ the following bounds hold 
 \be\label{ac2}
\min_{\bs: \|\bs\|_1=n} \sup_{f\in W^\br_{q,\alpha} }d(f,\Tr(R(\bs)))_q \le C(\br,q,d) 2^{-g(\br)n}.
 \ee
  \be\label{ac3}
\min_{\bs: \|\bs\|_1=n} \sup_{f\in H^\br_{q} }d(f,\Tr(R(\bs)))_q \le C(\br,q,d) 2^{-g(\br)n}.
 \ee
  
Note that for any collection $\cX$ and any function class $\bF$ we have
$$
\sup_{f\in \bF} \min_{L\in\cX} \|f-\ell(\xi,L)(f)\|_\infty \le  \min_{L\in\cX}\sup_{f\in \bF}  \|f-\ell(\xi,L)(f)\|_\infty.
$$
Therefore, Proposition \ref{acP1} and bounds (\ref{ac2}) and (\ref{ac3}) imply the following statement.

\begin{Proposition}\label{acP2} Let $n,d \in \N$.  There exists a set $\xi:= \{\xi^j\}_{j=1}^m \subset \T^d $, with $m \le C(d)2^n$, which provides    $L_\infty$-universal discretization (\ref{I3}) with $D=D(d)$ for the collection $H(n,d)$ and for   any  function $ f \in \C(\Omega)$, which belongs to either $W^\br_{\infty,\alpha}$ or $H^\br_\infty$, we have
 \be\label{ac1}
  \|f-\ell (\xi,H(n,d))(f)\|_\infty \le  C(\br,d)2^{-g(\br)n}.
 \ee
 \end{Proposition}

\section{Some other recovery operators}
\label{F}

In Sections \ref{ud} and \ref{ac} we discussed the recovery algorithm $\ell(\xi,\cX)$ and its special 
realization $\ell(\xi,H(n,d))$. Theorem \ref{udT1} provides the Lebesgue-type inequality for approximation by the algorithm $\ell(\xi,\cX)$. We proved that inequality in the case of the uniform norm. We do not have a similar inequality in the case of approximation in the $L_p$ norm with $p<\infty$. The algorithm $\ell(\xi,\cX)$ has two nonlinear steps of its realization. First, we apply 
the recovery algorithm $\ell(\xi,X(n))$, $X(n)\in\cX$. Second, we minimize the error of approximation 
over all subspaces $X(n)$. In this section we focus on the first step and discuss other algorithms, which are simpler than $\ell(\xi,X(n))$ but still provide good approximation. Here we only consider recovery 
of periodic functions from Sobolev and Nikol'skii classes discussed in Section \ref{ac}. We mostly concentrate on the case $d=2$, where the strongest results are obtained. Our arguments are based on 
the very recent paper \cite{VT190}. 

 {\bf Case $d=2$. Fibonacci points.} We need some classical trigonometric polynomials for our further argument (see \cite{Z} and \cite{VTbookMA}). We begin with the univariate case. 
 The Dirichlet kernel of order $j$:
$$
\mathcal D_j (x):= \sum_{|k|\le j}e^{ikx} = e^{-ijx} (e^{i(2j+1)x} - 1)
(e^{ix} - 1)^{-1} 
$$
$$
=\bigl(\sin (j + 1/2)x\bigr)\bigm/\sin (x/2)
$$
   is an even trigonometric polynomial.  The Fej\'er kernel of order $j - 1$:
$$
\mathcal K_{j} (x) := j^{-1}\sum_{k=0}^{j-1} \mathcal D_k (x) =
\sum_{|k|\le n} \bigl(1 - |k|/j\bigr) e^{ikx} 
$$
$$
=\bigl(\sin (jx/2)\bigr)^2\bigm /\bigl(j (\sin (x/2)\bigr)^2\bigr).
$$
The Fej\'er kernel is an even nonnegative trigonometric
polynomial of order $j-1$.  It satisfies the obvious relations
\be\label{FKm}
\| \mathcal K_{j} \|_1 = 1, \qquad \| \mathcal K_{j} \|_{\infty} = j.
\ee
The de la Vall\'ee Poussin kernel
\be\label{A2}
\mathcal V_{j} (x) := j^{-1}\sum_{l=j}^{2j-1} \mathcal D_l (x)= 2\cK_{2j}(x)-\cK_j(x) 
\ee
is an even trigonometric
polynomial of order $2j - 1$.

In the two-variate case  define the Fej\'er and de la Vall\'ee Poussin kernels as follows:
$$
\cK_\bj(\bx):=  \cK_{j_1}(x_1) \cK_{j_2}(x_2), \qquad \mathcal V_{\mathbf j} (\bx) := \cV_{j_1}(x_1) \cV_{j_2}(x_2)    ,\qquad
\mathbf j = (j_1,j_2) .
$$

Let $\{b_n\}_{n=0}^{\infty}$, $b_0=b_1 =1$, $b_n = b_{n-1}+b_{n-2}$,
$n\ge 2$, -- be the Fibonacci numbers.
Denote 
$$
\mathbf y^{\nu}:=\bigl(2\pi\nu/b_n, 2\pi\{\nu
b_{n-1}/b_n\}\bigr), \quad \nu = 1,\dots,b_n,\quad \cF_n:=\{\by^\nu\}_{\nu=1}^{b_n}.
$$
In this definition $\{a\}$ is the fractional part of the number $a$. The cardinality of the set $\cF_n$ is equal to $b_n$. 

For $N\in\N$ define the {\it hyperbolic cross} in dimension $2$ as follows:
$$
\Gamma(N):=\Gamma(N,2):= \left\{\bk\in\Z^2: \prod_{j=1}^2 \max(|k_j|,1) \le N\right\}.
$$
The following lemma is well known (see, for instance, \cite{VTbookMA}, p.274).

\begin{Lemma}\label{AL1} There exists an absolute constant $\gamma > 0$
such that for any $n > 2$ for the $2$-dimensional hyperbolic cross we have
  for any $f\in \Tr(\Gamma(N))$ with $N\le \gamma b_n$  
$$
b_n^{-1}\sum_{\nu=1}^{b_n}f\bigl(2\pi\nu/b_n,
2\pi\{\nu b_{n-1} /b_n \}\bigr) = (2\pi)^{-2}\int_{\T^2} f(\bx)d\bx.
$$
\end{Lemma}

 As above, for $\ba\in \bbC^m$ define the norm
$$
\|\ba\|_{p} := \left(\frac{1}{m}\sum_{i=1}^m |a_i|^p\right)^{1/p},\quad 1\le p<\infty; \quad \|\ba\|_\infty := \max_{i}|a_i|.
$$

  Theorem \ref{FT1} follows from the proof of Theorem 1.1 in \cite{VT190}. 

 \begin{Theorem}\label{FT1} Let $\gamma$ be from Lemma \ref{AL1}. For a given $n\in\N$ denote $n'\in\N$ to be the largest satisfying $2^{n'} \le \gamma b_n/9$. The Fibonacci point set $\cF_n$ provides the following two properties for the collection $H(n',2)$.  
 
{\bf (I).} For any $\bs$ satisfying $\|\bs\|_1 \le n'$ and any $f \in 
\Tr(R(\bs))$ we have $(2^\bs := (2^{s_1},2^{s_2}))$
$$
 f(\bx)=(2\pi)^{-2} \int_{\T^2} f(\by)\cV_{2^\bs}(\bx-\by)d\by = \frac{1}{b_n} \sum_{\nu=1}^{b_n} f(\by^\nu)\cV_{2^\bs}(\bx-\by^\nu).
$$

{\bf (II).}  For any $\bs$ satisfying $\|\bs\|_1 \le n'$  we have  
$$
\left\|\frac{1}{b_n}\sum_{\nu=1}^{b_n}  a_\nu|\cV_{2^{\bs}}(\bx-\by^\nu)|\right\|_\infty \le 9 \|\ba\|_{\infty},\quad \ba = (a_1,\dots,a_{b_n}).
$$
\end{Theorem}

We now define the recovery algorithm. First, we define for $f\in \cC(\T^2)$
$$
V_\bs(f) := V_\bs(\cF_n)(f) := \frac{1}{b_n}\sum_{\nu=1}^{b_n}  f(\by^\nu)\cV_{2^{\bs}}(\bx-\by^\nu).
$$
 It is a simple linear operator of discrete convolution. Second, we define
 $$
 \bs^o(f) := \text{arg}\min_{ \bs: \|\bs\|_1=n'}\|f-V_\bs(f)\|_\infty,
 $$
  \be\label{F1}
 V^n(f):= V_{\bs^o(f)}(f).
\ee
 
 We now prove the following analog of Proposition \ref{acP1}.
 
 \begin{Proposition}\label{FP1}  For the collection $H(n',2)$ and for   any  function $ f \in \cC(\T^2)$ we have
 \be\label{F3}
  \|f-V^n(f)\|_\infty \le  10 \min_{\bs: \|\bs\|_1=n'} d(f,\Tr(R(\bs)))_\infty.
 \ee
 \end{Proposition}
 \begin{proof} Let $\bs$ be such that $\|\bs\|_1=n'$. Then by property {\bf (I)} from Theorem \ref{FT1}
 we obtain that for any $t\in \Tr(R(\bs))$ 
 $$
 V_\bs(t) = t.
 $$
 By  property {\bf (II)} from Theorem \ref{FT1} we obtain
 $$
 \left\|\frac{1}{b_n}\sum_{\nu=1}^{b_n} | \cV_{2^{\bs}}(\bx-\by^\nu)|\right\|_\infty \le 9.
 $$
 Therefore, for any $g\in\cC(\T^2)$ we have
 $$
 \|V_\bs(g)\|_\infty \le 9\|g\|_\infty.
 $$
 Thus, for any $t\in \Tr(R(\bs))$ we have
 $$
 \|f-V_\bs(f)\|_\infty =  \|f-t-(V_\bs(f-t))\|_\infty\le 10\|f-t\|_\infty.
 $$
 Taking infimum over all $t\in \Tr(R(\bs))$ we obtain
 $$
  \|f-V_\bs(f)\|_\infty \le 10 d(f,\Tr(R(\bs))_\infty.
  $$
  This bound and the definition of $V^n(f)$ complete the proof.
 
 \end{proof}
 
 Proposition \ref{FP1} and bounds (\ref{ac2}), (\ref{ac3}) imply the following analog of Proposition \ref{acP2}. 
 
 \begin{Proposition}\label{FP2} Let $n \in \N$.   Then for   any  function $ f \in \C(\T^2)$, which belongs to either $W^\br_{\infty,\alpha}$ or $H^\br_\infty$, we have
 \be\label{F5}
  \|f-V^n(f)\|_\infty \le  C(\br)b_n^{-g(\br)}.
 \ee
 \end{Proposition}
 
 \begin{Remark}\label{FRF} The linear recovery operator $V^n$ only uses function values at $b_n$ points. It is known (see \cite{VTbookMA}, p.125) that for each individual class 
$W^\br_{\infty,\alpha}$ or $H^\br_\infty$ the error of linear recovery with $m$ function values 
cannot be better (in the sense of order) than $m^{-g(\br)}$. Thus, Proposition \ref{FP2} shows
that the operator $V^n$ provides optimal in the sense of order recovery for each class 
$W^\br_{\infty,\alpha}$ or $H^\br_\infty$.
\end{Remark}
 
 {\bf Case $d\ge 3$. Korobov points.} Here we extend the results of this section in the case $d=2$ to the case $d\ge 3$. Instead of the Fibonacci point sets we consider the Korobov point sets. We obtain results somewhat similar to those from above but not as sharp as results on the Fibonacci point sets (compare Remarks \ref{FRF} and \ref{FRK}). It is a well known phenomenon in numerical integration. We prove a conditional result under the assumption that the Korobov cubature formulas are exact on a certain subspace of trigonometric polynomials with frequencies  from a hyperbolic cross. 
 There are results that guarantee existence of such cubature formulas.  
 
  Let $m\in\N$, $\mathbf h := (h_1,\dots,h_d)$, $h_1,\dots,h_d\in\Z$.
We consider the cubature formulas
$$
P_m (f,\mathbf h):= m^{-1}\sum_{\nu=1}^{m}f\left ( 2\pi\left  \{\frac{\nu h_1}
{m}\right\},\dots, 2\pi\left \{\frac{\nu h_d}{m}\right\}\right),
$$
which are called the {\it Korobov cubature formulas}.  In the case $d=2$, $m=b_n$, $\mathbf h = (1,b_{n-1})$ we have
$$
P_m (f,\mathbf h) =  \frac{1}{b_n}\sum_{\by\in \cF_n} f(\by).
$$

Denote 
$$
\mathbf w^{\nu}:=\left ( 2\pi\left \{\frac{\nu h_1}
{m}\right\},\dots, 2\pi\left \{\frac{\nu h_d}{m}\right\}\right), \quad \nu = 1,\dots,m,\quad \cR_m(\bh):=\{\bw^\nu\}_{\nu=1}^m.
$$
The set $\cR_m(\bh)$ is called the {\it Korobov point set}.    

For $N\in\N$ define the {\it hyperbolic cross}   by
$$
\Gamma(N,d):= \left\{\bk= (k_1,\dots,k_d)\in\Z^d\colon \prod_{j=1}^d \max(|k_j|,1) \le N\right\}.
$$
Denote 
$$
\Tr(N,d) := \left\{f\, :\, f(\bx)= \sum_{\bk\in \Gamma(N,d)} c_\bk e^{i(\bk,\bx)}\right\}.
$$
 
\begin{Definition}\label{CD1} We say that the Korobov cubature formula $P_m(\cdot,\bh)$ is exact on $\Tr(N,d)$ if condition
\be\label{ex}
P_m(f,\bh) = (2\pi)^{-d}\int_{\T^d} f(\bx)d\bx, \quad \forall f\in \Tr(N,d),
\ee
   is satisfied.
\end{Definition} 

 {\bf Special Korobov point sets.} Let $N\in \N$ be given. Clearly, we are interested in as small $m$ as possible such that there exists a Korobov cubature formula, which is exact on $\Tr(N,d)$. In the case of $d=2$ the Fibonacci cubature formula is an ideal in a certain sense choice. 
There is no known Korobov cubature formulas in case $d\ge 3$, which are as good as  the Fibonacci cubature formula in case $d=2$. We now formulate some known results in this direction. Consider  a special case $\mathbf h = (1,h,h^2,\dots,h^{d-1})$, $h\in\N$. In this case we write in the notation of $\cR_m(\bh)$ and $P_m(\cdot,\bh)$ the scalar $h$ instead of the vector $\mathbf h$, namely, $\cR_m(h,d)$ and $P_m(\cdot,h,d)$. The following Lemma \ref{CL1} is a well known result (see, for instance \cite{VTbookMA}, p.285). 

\begin{Lemma}\label{CL1} Let $m$ and $N$ be a prime  
and a natural number, respectively, such that
\be\label{C3}
\bigl|\Gamma(N,d)\bigr| < (m-1)/d .
\ee
Then there is a natural number $h\in [1,m)$ such that for any $f\in \Tr(N,d)$ we have 
$$
P_m(f,h,d) = (2\pi)^{-d}\int_{\T^d} f(\bx)d\bx.
$$
\end{Lemma}

Note that the cardinality of $\Gamma(N,d)$ is of order $N(\log N)^{d-1}$ and, therefore, the largest $N$, satisfying (\ref{C3}), is of order $m(\log m)^{1-d}$. 

In the same way as Theorem \ref{FT1} was derived from Lemma \ref{AL1} in \cite{VT190}
the following Theorem \ref{CT1} can be derived from Definition \ref{CD1}. We do not present the proof here.  Lemma \ref{CL1} provides existence of special Korobov point sets satisfying Definition \ref{CD1}.  

Let $\cV_\bj(\bx):= \prod_{i=1}^d \cV_{j_i}(x_i)$ be the $d$-variate de la Vall\'ee Poussin kernels for $\bj = (j_1,\dots,j_d)$.

\begin{Theorem}\label{CT1} Let the Korobov cubature formula $P_m(\cdot,\bh)$ be exact on $\Tr(N,d)$ and let $\ell \in \N$ be the largest satisfying $2^\ell \le 3^{-d}N$. Then  the Korobov point set $\cR_m(\bh)$ provides the following two properties for the collection $H(\ell,d)$.  

{\bf (I).} For any $\bs\in\N^d$, satisfying $\|\bs\|_1 \le \ell$, and any $f \in 
\Tr(R(\bs))$ we have
$$
 f(\bx)=(2\pi)^{-d} \int_{\T^d} f(\by)\cV_{2^\bs}(\bx-\by)d\by = \frac{1}{m} \sum_{\nu=1}^{m} f(\bw^\nu)\cV_{2^\bs}(\bx-\bw^\nu).
$$

{\bf (II).}   
For any $\bs$ satisfying $\|\bs\|_1 \le \ell$   we have
$$
\left\|\frac{1}{m}\sum_{\nu=1}^{m}  a_\nu|\cV_{2^\bs}(\bx-\bw^\nu)|\right\|_\infty \le 3^d \|\ba\|_{\infty},\quad \ba = (a_1,\dots,a_{m}).
$$
\end{Theorem}

\begin{Remark}\label{CR1} Lemma \ref{CL1} implies that for any $N\in\N$ there exist $\bh$ 
and $m\le C(d)N(\log N)^{d-1}$ with some positive $C(d)$ such that statements (I) and (II) of Theorem \ref{CT1} hold.
\end{Remark}

 We now define the recovery algorithm based on the Korobov point set $\cR_m(\bh)$ such that $P_m(\cdot,\bh)$ is exact on $\Tr(N,d)$. First, we define for $f\in \cC(\T^d)$
$$
V_\bs(f) := V_\bs(\cR_m(\bh))(f) := \frac{1}{m}\sum_{\nu=1}^{m}  f(\bw^\nu)\cV_{2^{\bs}}(\bx-\bw^\nu).
$$
 It is a simple linear operator of discrete convolution. Second, we define
 $$
 \bs^o(f) := \text{arg}\min_{ \bs: \|\bs\|_1=\ell}\|f-V_\bs(f)\|_\infty,
 $$
 where $\ell$ is from Theorem \ref{CT1},
  \be\label{C1}
 V^\ell(f):= V_{\bs^o(f)}(f).
\ee
 
 The following Proposition \ref{CP1} is an analog of Proposition \ref{FP1}. We do not present its proof 
 here, which goes along the line of the proof of Proposition \ref{FP1}.
 
 \begin{Proposition}\label{CP1}  For the collection $H(\ell,d)$ and for   any  function $ f \in \cC(\T^2)$ we have
 \be\label{C5}
  \|f-V^\ell(f)\|_\infty \le  (3^d+1) \min_{\bs: \|\bs\|_1=\ell} d(f,\Tr(R(\bs)))_\infty.
 \ee
 \end{Proposition}

 Proposition \ref{CP1} and bounds (\ref{ac2}), (\ref{ac3}) imply the following analog of Proposition \ref{acP2}. 
 
 \begin{Proposition}\label{CP2} Let $n,d \in \N$.   Then for   any  function $ f \in \C(\T^d)$, which belongs to either $W^\br_{\infty,\alpha}$ or $H^\br_\infty$, we have
 \be\label{C6}
  \|f-V^\ell(f)\|_\infty \le  C(\br,d)N^{-g(\br)}.
 \ee
 \end{Proposition}

  \begin{Remark}\label{FRK} By Remark \ref{CR1} we can choose the linear recovery operator $V^\ell$ in such a way that it only uses function values at $m\le C(d)N(\log N)^{d-1}$ points. It is known (see \cite{VTbookMA}, p.125) that for each individual class 
$W^\br_{\infty,\alpha}$ or $H^\br_\infty$ the error of linear recovery with $m$ function values 
cannot be better (in the sense of order) than $m^{-g(\br)}$. Thus, Proposition \ref{CP2} shows
that the operator $V^\ell$ provides suboptimal   in the sense of order (up to the $(\log N)^c$ factors) recovery for each class 
$W^\br_{\infty,\alpha}$ or $H^\br_\infty$.
\end{Remark}

 {\bf Acknowledgement.}
The author would like to thank E. Kosov for  careful reading of the paper and for helpful  suggestions and comments.

  \Addresses

\end{document}